\documentclass[12pt, reqno]{amsart}
\usepackage{amsmath, amsthm, amscd, amsfonts, amssymb, graphicx, color}
\usepackage{mathrsfs}
\usepackage{enumerate}
\usepackage[bookmarksnumbered, colorlinks, plainpages]{hyperref}
\hypersetup{colorlinks=true,linkcolor=red, anchorcolor=green, citecolor=cyan, urlcolor=red, filecolor=magenta, pdftoolbar=true}

\newtheorem{theorem}{Theorem}[section]
\newtheorem{lemma}[theorem]{Lemma}

\newtheorem{cor}[theorem]{Corollary}

\newtheorem{example}[theorem]{Example}

\theoremstyle{remark}
\newtheorem{remark}[theorem]{\bf{Remark}}
\numberwithin{equation}{section}
\allowdisplaybreaks
\begin{document}

\title [Improved Bounds for  numerical radius  and {\large$\lowercase{a}$-}numerical radius  \ldots ]{Improved Bounds for  numerical radius  and {\Large$\lowercase{a}$-}numerical radius   in ${C}^*$-algebras }
	\author[ S. Mahapatra, A. Ghosh,  and  R. Birbonshi ]{ Saikat Mahapatra, Arobinda Ghosh  and  Riddhick Birbonshi}

\address[Mahapatra]{Department of Mathematics, Jadavpur University, Kolkata 700032, West Bengal, India}
\email{smpatra.lal2@gmail.com}

    \address[Ghosh]{Department of Mathematics, Jadavpur University, Kolkata 700032, West Bengal, India}
\email{7aghosh@gmail.com}

\address[Birbonshi] {Department of Mathematics, Jadavpur University, Kolkata 700032, West Bengal, India}
\email{riddhick.math@gmail.com}



\maketitle	

\begin{abstract} 
		In this article,  we derive several significant upper bounds for the numerical radius and $a$-numerical radius of an element in a ${C}^*$-algebra by improving inequalities for positive linear functionals. Our findings refine and generalize the existing inequalities. Furthermore, we introduce a new notion to derive improved upper bounds of the numerical radius for an element in a ${C}^*$-algebra using the Moore-Penrose inverse.

\end{abstract}

\section{Introduction}
Let $\mathfrak{A} $ be a complex unital  ${C}^*$-algebra, where the unit is denoted by  $\textbf{1}_{\mathfrak{A}}$ and the dual space of $\mathfrak{A}$ is represented by $\mathfrak{A}^*$. 
A linear functional $\phi\in \mathfrak{A}^*$ is called  positive and denoted by $\phi\ge 0$ if $ \phi(x^*x)\ge 0$ for all $x\in \mathfrak{A}$.  Any such positive linear functional also satisfies the relation $\phi(x^*)=\overline{\phi(x)}$ for any $x\in \mathfrak{A}$. Moreover,
if $\phi$ is a positive linear functional on $\mathfrak{A} $ then for all $x,y\in \mathfrak{A}$, $|\phi(x^*y)|^2\le \phi(x^*x)\phi(y^*y)$, which is called Cauchy-Schwarz inequality on a ${C}^*$-algebra (see \cite{book1,putman}). A positive linear functional on $\mathfrak{A} $ is called  state if $\phi(\textbf{1}_{\mathfrak{A}})=1$. The collection of all states on $\mathfrak{A} $  is denoted by 
$\mathcal{S}(\mathfrak{A})=\{\phi\in\mathfrak{A}^* : \phi\ge 0 \,\, \mbox{and}\,\, \phi(\textbf{1}_{\mathfrak{A}})=1 \}$. For an element $x\in \mathfrak{A} $, the algebraic numerical range  and the algebraic numerical radius are, respectively, given by \begin{align*}
V(x)=\{\phi(x):\phi\in \mathcal{S}\mathfrak{(A)}\}\,\,\mbox{ and} \,\, v(x)=\sup\{|\lambda|: \lambda\in V(x)\}.   
\end{align*}

Let $\mathscr{B}(\mathcal{H})$ be the unital ${C}^*$-algebra consisting of all bounded linear operators on a complex separable Hilbert space $\big( \mathcal{H},\langle \cdot \rangle \big)$. In particular, when $\mathfrak{A}=\mathscr{B}(\mathcal{H})$   and $T\in \mathscr{B}(\mathcal{H})$, the spatial numerical range of $T$ is defined as $W(T)=\big\{ \big \langle Tx,x \big\rangle: x\in \mathcal{H},\|x\|=1\big\}$, and its spatial numerical radius is defined as $w(T)=\sup\{|\lambda|: \lambda\in W(T)\}$. The algebraic numerical range $V(T)$ is the closure of the spatial numerical range  $W(T)$.

For any $x\in \mathfrak{A}$, the fundamental inequality between  ${C}^*$-norm and the numerical radius $v(x)$   is as follows (see \cite{bonsall}):
\begin{eqnarray*}
    \frac{1}{2}\|x\|\le v(x) \le \|x\|. \label{in1}
\end{eqnarray*}
Throughout this paper, we assume $a$ is a non-zero positive element in $\mathfrak{A}$. Now, we
define  $\mathcal{S}_{a}\mathfrak{(A)}=\{\varphi\in \mathfrak{A}^{*}:\varphi\ge0,\hspace{0.07 cm}\varphi(a)=1\}$. Alternatively, this collection can be described as
$ \mathcal{S}_{a}\mathfrak{(A)}= \big\{\frac{\phi}{\phi(a)}:\phi\in \mathcal{S}(\mathfrak{A}),\hspace{0.07 cm} \phi(a)\neq 0\big\}$. Recently, Mabrouk and Bourhim  have introduced and studied the concepts of the 
$a$-numerical range and 
$a$-numerical radius for elements in unital 
 ${C}^*$-algebras (see \cite{bourhim2021numerical}).
 The  $a$-numerical range of  an element $x\in \mathfrak{A} $  is 
$V_{a}(x)=\{\varphi(ax):\varphi\in \mathcal{S}_{a}\mathfrak{(A)}\},\hspace{0.15 cm}
  $ and the $a$-numerical radius is defined by $ v_{a}(x)=\sup\{|z|:z\in V_{a}(x)\}. $
Specifically, for $a=\textbf{1}_{\mathfrak{A}}$, the $a$-numerical range and $a$-numerical radius of $x$ become algebraic numerical range and algebraic numerical radius of $x$, respectively; i.e., $V_a(x)=V(x)$ and 
$v_a(x)=v(x)$. These concepts have been introduced in  \cite{bourhim2021numerical} which generalize the concepts of the spatial  $A$-numerical range and spatial $A$-numerical radius of a bounded linear operator $T\in \mathscr{B}(\mathcal{H})$, which are given by  $W_A(T)=\big\{ \big \langle Tx,x \big\rangle_A: x\in \mathcal{H},\|x\|_A=1\big\}$ and $w_A(T)=\sup \{|z|: z\in W_A(T)\}$, respectively. Here $A$ is a positive bounded linear operator  on a Hilbert space $(\mathcal{H},\langle \cdot\rangle)$ and the $ A $-semi norm of any $x\in \mathcal{H}$  is defined by $\|x\|_A^2=\langle x,x \rangle_A=\langle Ax,x\rangle$. For an element $x\in \mathfrak{A}$, the $a$-Crawford number of $x$, denoted by $c_{a}(x)$, is defined as follows:
\[c_{a}(x)=\inf\{|z|:z\in V_a(x)\}.\] 
In particular, when $a=\textbf{1}_{\mathfrak{A}}$, this reduces to the usual Crawford number of $x$, written as $c(x)$, (see \cite{bonsall,azamanich}).

 For any element $x\in \mathfrak{A}$, define $\|x\|_{a}=\sup \{\sqrt{\varphi(x^{*}ax}):\varphi\in \mathcal{S}_{a}\mathfrak{(A)}\}$.
 The value of  $\|x\|_{a} $ equals zero if and only if $ax=0 $.  Moreover, for some $x\in \mathfrak{A}$, 
$\|x\|_a$ may be infinite (see \cite{bourhim2021numerical}).
Let $\mathfrak{A}^a$ be the collection of  such elements $x\in \mathfrak{A}$ for which $\|x\|_a< \infty.$ The function $\|\cdot\|_a$ defines a semi-norm on $\mathfrak{A}^a$ and also satisfies the condition  $  \|xy\|_a\le\|x\|_a\|y\|_a$   for all $x,y\in \mathfrak{A}^a$. Thus $\mathfrak{A}^a$
 forms a subalgebra of $\mathfrak{A}$. 
For any $x\in \mathfrak{A}$, an element $x^{\#_a}\in \mathfrak{A}$ is called  $a$-adjoint of $x$ if it satisfies the relation $x^*a=ax^{\#_a}.$ An
$a$-adjoint for a given 
$x\in \mathfrak{A}$ may not always exist or be unique.
The set of all elements in  $\mathfrak{A}$  that possess $a$-adjoints is denoted by $\mathfrak{A}_a$, i.e., 
\[\mathfrak{A}_a=\{ x\in \mathfrak{A}:\mbox{there is $x^{\#_a}\in \mathfrak{A}$ such  that }  x^*a=ax^{\#_a} \}. \]
Moreover, $\mathfrak{A}_a$ forms a subalgebra of $\mathfrak{A}$, and it is also contained in   $\mathfrak{A}^a$.
If $x\in \mathfrak{A}_a$ and $x^{\#_a}$ is an $a$-adjoint of $x$, then by \cite[Corollary 4.9]{bourhim2021numerical}\begin{equation*}\|x\|_a^2 =\|x^{\#_a}\|_a^2=\|x^{\#_a}x\|_a=\|xx^{\#_a}\|_a. \end{equation*}
An element $x \in \mathfrak{A}$ is called $a$-self- adjoint if $ax$ is self-adjoint, i.e.,  $ax=x^*a$. If $x\in \mathfrak{A}$, then for any natural number $n$, it follows that  $v_a(x^n) \le v_a^n(x)$. 
In \cite{bourhim2021numerical}, Bourhim and Mabrouk have shown that for any $x\in \mathfrak{A}_a$, the inequality between the semi-norm  $\|\cdot \|_a$ and $a$-numerical radius $v_a(x)$ holds as below:
\begin{eqnarray}
    \frac{1}{2}\|x\|_a\le v_a(x) \le \|x\|_a. \label{in2}
\end{eqnarray} The left-hand inequality becomes equality if $ax\neq 0$ and $ax^2=0$, and the right-hand inequality becomes equality if $x$ is $a$-self-adjoint.

In \cite{mabrouk},  Zamani and Mabrouk  have also provided the following upper bound for the $a$-numerical radius of $x\in \mathfrak{A}_a$:
\begin{eqnarray}
v_a^2(x)&\le& \frac{1}{2}\|x^{\#_a}x+xx^{\#_a}\|_a. \label{in33}  
\end{eqnarray}
Clearly, the inequality in (\ref{in33})  provides a sharper bound than the right-hand side inequality in (\ref{in2}).
Over the years,   $a$-numerical range and numerical range of elements in unital ${C}^*$-algebras have been studied by many mathematicians,  and we refer to   \cite{mabrouk2,pintu,bourhim2021numerical,bourhim2,chen,chen2,jana,mabrouk,smpatra,azamanich,newzamani}.

An element $x\in \mathfrak{A}$ is called regular if it has a generalized inverse in  $\mathfrak{A}$,   i.e., if there exists $x^{\prime} \in \mathfrak{A}$ for which $x = x x' x$.  It is easy to verify that the elements $xx^{\prime}$ and $x^{\prime}x$ are idempotents is $\mathfrak{A}$.  A generalized inverse $x^{\prime}$ of a regular element $x\in \mathfrak{A}$ is said to be normalized if it satisfies both $x=xx^{\prime}x$ and $x^{\prime}=x^{\prime}xx^{\prime}$. In the presence of an involution $*:\mathfrak{A}\to \mathfrak{A}$ it is possible to enquire if $xx^{\prime}$ and $x^{\prime}x$ are self-adjoint, equivalently whether or not $(xx^{\prime})^*=xx^{\prime}$ and $(x^{\prime}x)^*=x^{\prime}x$. In this case, $x^{\prime}$ is called the Moore-Penrose inverse of $x$, and it is denoted by $x^{\dag}$.

In \cite{robin}, it was proved that every regular element 
$x$ in a 
${C}^*$-algebra 
 $\mathfrak{A}$ possesses a Moore-Penrose inverse, and this Moore-Penrose inverse is unique. Thus, for any regular element 
$x\in \mathfrak{A}$, the Moore-Penrose inverse, denoted by 
$x^{\dag}$, is the unique element in 
$\mathfrak{A}$ that satisfies the following equations:
\begin{align*}
(i)~~ x^{\dag} = x^{\dag} x x^{\dag},\
(ii)~~ x = x x^{\dag} x,\
(iii)~~ (x^{\dag} x)^* = x^{\dag} x,\
(iv)~~ (x x^{\dag})^* = x x^{\dag}.
\end{align*}
According to the uniqueness of the Moore-Penrose inverse of a regular element $x$, $x^*$ also has a Moore-Penrose inverse and $(x^*)^{\dag}=(x^{\dag})^*$. Furthermore, if $x$ is a regular element, then $x^{\dag}$  is also  regular and $(x^{\dag)^{\dag}}=x$. For additional information regarding the Moore-Penrose inverse of a regular element $x\in\mathfrak{A}$, we refer to \cite{Boasso,robin2,robin}.

The purpose of this paper is to derive new upper bounds for $a$-numerical radius and numerical radius of  elements in a unital
${C}^*$-algebra. In Section \ref{sec2}, we review some well-known inequalities and derive new ones for positive linear functionals that are crucial in our subsequent proofs. Section \ref{sec3} is devoted to establish several new upper bounds for $a$-numerical radius inequalities for elements in $\mathfrak{A}_a$, which improve previously known results.  Lastly,  Section \ref{sec4} presents some new upper bounds of the numerical radius inequalities for an element in $\mathfrak{A}$ using the Moore-Penrose inverse.  We show that the results presented here refine and generalize
several existing findings.

\section{Preliminaries}\label{sec2} We start this section with several supporting lemmas that will be referenced frequently throughout our main results.

\begin{lemma}\cite[Th. 10.1]{inequality}(Weighted Cauchy-Schwarz inequality)\label{vv2}
Let $\alpha_j$ and $ \beta_j $ 
 ($j=1,2,\cdots,k$) be real numbers and let $w_j$ ($j=1,2,\cdots,k$) be  positive real numbers. Then the following inequality holds:  
 \[\bigg (\sum_{j=1}^kw_j\alpha_j\beta_j\bigg)^2\le \bigg( \sum_{j=1}^kw_j\alpha_j^2\bigg)\bigg( \sum_{j=1}^kw_j\beta_j^2\bigg).\]
\end{lemma}
\begin{lemma}\cite{hardy}\label{young}
     Let $a,b\geq 0$, $0\leq \alpha\leq 1$ and $p,q>1$ such that $\frac{1}{p}+\frac{1}{q}=1$. Then for $r\geq 1$,\\
   $(i)$ $a^\alpha b^{1-\alpha}\leq\alpha a+(1-\alpha)b\leq(\alpha a^r+(1-\alpha)b^r)^\frac{1}{r}$,\\
         $(ii)$  $ab\leq\frac{a^p}{p}+\frac{b^q}{q}\leq\left(\frac{a^{pr}}{p}+\frac{b^{qr}}{q}\right)^\frac{1}{r}$.
\end{lemma}


Let $\mathcal{P}_k$ be the set of all probability distributions over $k$ elements. If $\textbf{P} = (p_1, p_2, \cdots, p_k) \in \mathcal{P}_k$, then  $p_j \ge 0$ for every $j \in\{ 1, 2, \cdots, k\}$, and $\sum_{j=1}^k p_j = 1$.

\begin{lemma}\label{newcauchy} Let $x_j$, $y_j\in \mathfrak{A}$, $j=1,2,\cdots, k$ and $\mu\ge 0$.  If $\textbf{P}=(p_1,p_2,\cdots,p_k)\in \mathcal{P}_k$ is a probability distribution
 and $\varphi$ be a positive linear functional on $\mathfrak{A}$, then 
 \begin{eqnarray*}
\Big|\sum_{j=1}^{k}p_j\varphi(x^*_jay_j)\Big|^{2}&\le&\frac{1}{1+\mu} \Big|\sum_{j=1}^{k}p_j\varphi(x^*_jay_j) \Big|\bigg(\sum_{j=1}^{k}p_j\varphi(x^*_jax_j)\sum_{j=1}^{k}p_j\varphi(y^*_jay_j)\bigg)^{\frac{1}{2}}\nonumber\\
  &&+\frac{\mu}{1+\mu}\bigg(\sum_{j=1}^{k}p_j\sqrt{\varphi(x^*_jax_j)\varphi(y^*_jay_j)}\bigg)^2.
    \end{eqnarray*}
    \end{lemma}
 \begin{proof} Let $x_j, y_j\in \mathfrak{A} $ and  $\varphi$ be a positive linear functional on $\mathfrak{A}$. Then we have 
\begin{eqnarray}
     && \Big|\sum_{j=1}^{k}p_j\varphi(x^*_jay_j) \Big|^2\nonumber\\
      &=& \frac{1}{1+\mu}  \Big|\sum_{j=1}^{k}p_j\varphi(x^*_jay_j) \Big|^2+\frac{\mu}{1+\mu}  \Big|\sum_{j=1}^{k}p_j\varphi(x^*_jay_j) \Big|^2\nonumber\\
    &\le& \frac{1}{1+\mu}  \Big|\sum_{j=1}^{k}p_j\varphi(x^*_jay_j) \Big|^2+\frac{\mu}{1+\mu} \bigg(\sum_{j=1}^{k}p_j \Big|\varphi(x^*_jay_j) \Big|\bigg)^2\nonumber\\
      &\le&\frac{1}{1+\mu}\Big|\sum_{j=1}^{k}p_j\varphi(x^*_jay_j)\Big|^2+\frac{\mu}{1+\mu}\bigg(\sum_{j=1}^{k}p_j\sqrt{\varphi(x^*_jax_j)}\sqrt{\varphi(y^*_jay_j)}\bigg)^2.\label{sun}
\end{eqnarray}

This indicates that
\begin{eqnarray*}
    \Big|\sum_{j=1}^{k}p_j\varphi(x^*_jay_j) \Big|^2
         &\le& \frac{1}{1+\mu}  \Big|\sum_{j=1}^{k}p_j\varphi(x^*_jay_j) \Big| \bigg(\sum_{j=1}^{k}p_j \Big|\varphi(x^*_jay_j) \Big|\bigg)\\
         &&+\frac{\mu}{1+\mu}\bigg(\sum_{j=1}^{k}p_j\sqrt{\varphi(x^*_jax_j)}\sqrt{\varphi(y^*_jay_j)}\bigg)^2\nonumber\\
       &\le&\frac{1}{1+\mu}\Big|\sum_{j=1}^{k}p_j\varphi(x^*_jay_j)\Big|\bigg(\sum_{j=1}^{k}p_j\sqrt{\varphi(x^*_jax_j)}\sqrt{\varphi(y^*_jay_j)}\bigg)\nonumber\\   
  &&+\frac{\mu}{1+\mu}\bigg(\sum_{j=1}^{k}p_j\sqrt{\varphi(x^*_jax_j)}\sqrt{\varphi(y^*_jay_j)}\bigg)^2.\nonumber
\end{eqnarray*}
Now, using Lemma \ref{vv2}  we get        \begin{eqnarray*}
      \Big|\sum_{j=1}^{k}p_j\varphi(x^*_jay_j) \Big|^2
    &\le&\frac{1}{1+\mu} \Big|\sum_{j=1}^{k}p_j\varphi(x^*_jay_j) \Big|\bigg(\sum_{j=1}^{k}p_j\varphi(x^*_jax_j)\sum_{j=1}^{k}p_j\varphi(y^*_jay_j)\bigg)^{\frac{1}{2}}\nonumber\\
  &&+\frac{\mu}{1+\mu}\bigg(\sum_{j=1}^{k}p_j\sqrt{\varphi(x^*_jax_j)}\sqrt{\varphi(y^*_jay_j)}\bigg)^{2}.\nonumber
    \end{eqnarray*} Hence  the result follows.
    \end{proof}
\begin{remark} Taking $k=1$  in  Lemma \ref{newcauchy}, we get 
 \begin{align*}
 \Big|\varphi(x^*ay)\Big|^{2}
 &\le\frac{1}{1+\mu} \Big|\varphi(x^*ay) \Big|\big(\varphi(x^*ax)\varphi(y^*ay)\big)^{\frac{1}{2}}+\frac{\mu}{1+\mu}\varphi(x^*ax)\varphi(y^*ay)\\
 &\le \varphi(x^*ax)\varphi(y^*ay). 
    \end{align*}
    In particular, when $a=\textbf{1}_{\mathfrak{A}}$, the above inequality provides a refinement of the Cauchy–Schwarz inequality. 
\end{remark}

\begin{lemma}\cite[Lemma 2.9]{smpatra}\label{zzxx}
    Let $x_j$, $y_j\in \mathfrak{A}$, $j=1,2,\cdots, k$.  If $\textbf{P}=(p_1,p_2,\cdots,p_k)\in \mathcal{P}_k$ is a probability distribution
 and $\varphi\in \mathcal{S}_a(\mathfrak{A})$, then 
 \begin{eqnarray*}
     &&\Big| \sum_{j=1}^{k}p_j\varphi(ax_j)\Big|   \Big| \sum_{j=1}^{k}p_j\varphi(y_j^*a)\Big|\\
        &\le& \frac{ 1}{2}\Bigg( \Big| \sum_{j=1}^{k}p_j\varphi(y_j^*ax_j)\Big|+\bigg(  \sum_{j=1}^{k}p_j\varphi(x_j^*ax_j)  \sum_{j=1}^{k}p_j\varphi(y_j^*ay_j)\bigg)^{\frac{1}{2}}\Bigg).
 \end{eqnarray*}
\end{lemma}
Using this lemma, we obtain the following results.

\begin{lemma}\label{firsth}
 Let $x_j$, $y_j\in \mathfrak{A}$, $j=1,2,\cdots, k$ and $\mu\ge 0$.   If $\textbf{P}=(p_1,p_2,\cdots,p_k)\in \mathcal{P}_k$ is a probability distribution
 and $\varphi\in \mathcal{S}_a(\mathfrak{A})$, then 
 \begin{align*}
&\left(\Big|\sum_{j=1}^{k}p_j\varphi(ax_j)\Big|\Big|\sum_{j=1}^{k}p_j\varphi(ay_j)\Big|\right)^{2}\\&\le\frac{1}{4}U+\frac{3+2\mu}{4(1+\mu)}V+\frac{\mu}{4(1+\mu)}\bigg(\sum_{j=1}^kp_j\frac{\varphi(y^{*}_jay_j+x^{*}_jax_j)}{2}\bigg)^{2}, 
 \end{align*} where \begin{align*}
U&=\sum_{j=1}^{k}p_j\varphi(x^{*}_jax_j)\sum_{j=1}^{k}p_j\varphi(y^{*}_jay_j)\\
V&=\Big(\sum_{j=1}^{k}p_j\varphi(x^{*}_jax_j)\sum_{j=1}^{k}p_j\varphi(y^{*}_jay_j)\Big)^{\frac{1}{2}}\Big|\sum_{j=1}^{k}p_j\varphi(x^*_jay_j)\Big|.
 \end{align*}
\end{lemma}
\begin{proof} Let $x_j, y_j\in \mathfrak{A} $ and $\varphi\in \mathcal{S}_a(\mathfrak{A})$. Now, Lemma \ref{zzxx} yields
\begin{align*}
\Big(\Big|\sum_{j=1}^{k}p_j\varphi(ax_j)\Big|\Big|\sum_{j=1}^{k}p_j\varphi(ay_j)\Big|\Big)^2\le \frac{1}{4}U+\frac{1}{2}V+ \frac{1}{4}\Big|\sum_{j=1}^{k}p_j\varphi(x^*_jay_j)\Big|^2 .    
\end{align*}
       Now, applying Lemma \ref{newcauchy}, we get
    \begin{align*}&\left(\Big|\sum_{j=1}^{k}p_j\varphi(ax_j)\Big|\Big|\sum_{j=1}^{k}p_j\varphi(ay_j)\Big|\right)^2
\\&\le\frac{1}{4}U+\frac{3+2\mu}{4(1+\mu)}V+\frac{\mu}{4(1+\mu)}\bigg(\sum_{j=1}^kp_j\Big(\varphi(y^{*}_jay_j)\varphi(x^{*}_jax_j)\Big)^{\frac{1}{2}}\bigg)^{2}\\
&\le\frac{1}{4}U+\frac{3+2\mu}{4(1+\mu)}V+\frac{\mu}{4(1+\mu)}\bigg(\sum_{j=1}^kp_j\frac{\varphi\big(y^{*}_jay_j+x^{*}_jax_j\big)}{2}\bigg)^{2}.
    \end{align*} 
   This completes the proof.
\end{proof}

The next corollary follows directly from Lemma \ref{firsth}.
\begin{cor}\label{klj1}
     Let $x_j$, $y_j\in \mathfrak{A}_a$, $j=1,2,\cdots, k$ and $\mu\ge 0$.   If $\textbf{P}=(p_1,p_2,\cdots,p_k)\in \mathcal{P}_k$ is a probability distribution
 and $\varphi\in \mathcal{S}_a(\mathfrak{A})$, then 
 \begin{align*}
&\left(\big|\sum_{j=1}^{k}p_j\varphi(ax_j)\big|\big|\sum_{j=1}^{k}p_j\varphi(ay_j)\big|\right)^{2}\\
     \le&\frac{1}{4}\Big(\sum_{j=1}^{k}p_j\varphi(ax^{\#_a}_jx_j)\sum_{j=1}^{k}p_j\varphi(ay_jy^{\#_a}_j)\Big)+\frac{\mu}{4(1+\mu)}\bigg(\sum_{j=1}^kp_j\frac{\varphi(a(y_jy^{\#_a}_j+x^{\#_a}_jx_j))}{2}\bigg)^{2}\\&+\frac{3+2\mu}{4(1+\mu)}\Big(\sum_{j=1}^{k}p_j\varphi(ax^{\#_a}_jx_j)\sum_{j=1}^{k}p_j\varphi(ay_jy^{\#_a}_j)\Big)^{\frac{1}{2}}\Big|\sum_{j=1}^{k}p_j\varphi(ay_jx_j)\Big|.
 \end{align*} 
\end{cor}
\begin{proof}
    Since $y_j\in \mathfrak{A}_a$, there exists $y_j^{\#_a}\in \mathfrak{A} $ such that $ay_j^{\#_a}=y_j^*a$ for all $1\le j\le k$. Substituting $y_j=y_j^{\#_a}$ in Lemma \ref{firsth} and using the fact $\varphi(z^*)=\overline{\varphi(z)}$ for any $z\in\mathfrak{A}$, we obtain the desired result.
\end{proof}
\begin{lemma}\label{bads} Let $x_j$, $y_j\in \mathfrak{A}$, $j=1,2,\cdots, k$ and $\mu\ge 0$.   If $\textbf{P}=(p_1,p_2,\cdots,p_k)\in \mathcal{P}_k$ is a probability distribution
 and $\varphi\in \mathcal{S}_a(\mathfrak{A})$, then 
\begin{align*}
&  \bigg( \Big|\sum_{j=1}^kp_j\varphi(ax_j)\Big|   \Big| \sum_{j=1}^kp_j\varphi(ay_j)\Big|\bigg)^{2}\\
\le& \frac{1}{2(1+\mu)} \Big|\sum_{j=1}^{k}p_j\varphi(x^*_jay_j) \Big|^{2}+\frac{\mu}{2(1+\mu)}\bigg(\sum_{j=1}^{k}p_j\frac{\varphi(x^*_jax_j+y^*_jay_j)}{2}\bigg)^{2}\\
  &+  \frac{1}{2}\Big(\sum_{j=1}^kp_j\varphi(x^*_jax_j)\sum_{j=1}^kp_j  \varphi(y^*_jay_j)\Big).
\end{align*}
\end{lemma}
\begin{proof} Let $x_j, y_j\in \mathfrak{A} $ and $\varphi\in \mathcal{S}_a(\mathfrak{A})$. Then we have
      \begin{align*}
      & \bigg( \Big|\sum_{j=1}^kp_j\varphi(ax_j)\Big|  \Big | \sum_{j=1}^kp_j\varphi(ay_j)\Big|\bigg)^2\\
       &= \bigg( \Big|\sum_{j=1}^kp_j\varphi(ax_j)\Big|  \Big | \sum_{j=1}^kp_j\varphi(y_j^*a)\Big|\bigg)^2\\
        &\le \frac{  1}{4}\bigg(\Big| \sum_{j=1}^kp_j\varphi(x^*_jay_j)\Big|+\Big( \sum_{j=1}^kp_j \varphi(x^*_jax_j)  \sum_{j=1}^kp_j\varphi(y^*_jay_j)\Big)^{\frac{1}{2}}\bigg)^2\\
           &\le \frac{  1}{2}\Big| \sum_{j=1}^kp_j\varphi(x^*_jay_j)\Big|^2+\frac{1}{2}\Big( \sum_{j=1}^kp_j \varphi(x^*_jax_j)  \sum_{j=1}^kp_j\varphi(y^*_jay_j)\Big)\\
      &\le \frac{1}{2(1+\mu)} \Big|\sum_{j=1}^{k}p_j\varphi(x^*_jay_j) \Big|^{2}+\frac{\mu}{2(1+\mu)}\bigg(\sum_{j=1}^{k}p_j\sqrt{\varphi(x^*_jax_j)}\sqrt{\varphi(y^*_jay_j)}\bigg)^{2}\\
  &\,\,\,\,\,\,+  \frac{1
  }{2}\Big(\sum_{j=1}^kp_j\varphi(x^*_jax_j)\sum_{j=1}^kp_j  \varphi(y^*_jay_j)\Big)\\
  &\le\frac{1}{2(1+\mu)} \Big|\sum_{j=1}^{k}p_j\varphi(x^*_jay_j) \Big|^{2}+\frac{\mu}{2(1+\mu)}\bigg(\sum_{j=1}^{k}p_j\frac{\varphi(x^*_jax_j+y^*_jay_j)}{2}\bigg)^{2}\\
  &\,\,\,\,\,\,+  \frac{1}{2}\Big(\sum_{j=1}^kp_j\varphi(x^*_jax_j)\sum_{j=1}^kp_j  \varphi(y^*_jay_j)\Big),
    \end{align*} 
    where the first inequality is derived from Lemma \ref{zzxx}, while the second is obtained using Lemma \ref{vv2}. The third inequality comes from the inequality (\ref{sun}), and the last inequality follows from $ab\le \frac{a^2+b^2}{2}$ for all $a,b\ge 0$. 
     Hence, the result follows.
\end{proof}

The following corollary is an immediate consequence of Lemma \ref{bads}.
\begin{cor}\label{2ndcor}
    Let $x_j$, $y_j\in \mathfrak{A}_a$, $j=1,2,\cdots, k$ and $\mu\ge 0$.  If $\textbf{P}=(p_1,p_2,\cdots,p_k)\in \mathcal{P}_k$ is a probability distribution
 and $\varphi\in \mathcal{S}_a(\mathfrak{A})$, then
\begin{align*}
&  \bigg(\Big |\sum_{j=1}^kp_j\varphi(ax_j)\Big|  \Big | \sum_{j=1}^kp_j\varphi(ay_j)\Big|\bigg)^{2}\\
   &\le \frac{1}{2(1+\mu)} \Big|\sum_{j=1}^{k}p_j\varphi(ay_jx_j) \Big|^{2}+\frac{\mu}{2(1+\mu)}\bigg(\sum_{j=1}^{k}p_j\frac{\varphi\big(a(x^{\#_a}_jx_j+y_jy^{\#_a}_j)\big)}{2}\bigg)^{2}\\
  &\,\,\,\,\,\,+  \frac{1}{2}\Big(\sum_{j=1}^kp_j\varphi(ax^{\#_a}_jx_j)\sum_{j=1}^kp_j  \varphi(ay_jy^{\#_a}_j)\Big).
\end{align*}
\end{cor}
\begin{proof}
    Since $y_j\in \mathfrak{A}_a$, there exists $y_j^{\#_a}\in \mathfrak{A} $ such that $ay_j^{\#_a}=y_j^*a$ for all $1\le j\le k$. Substituting $y_j=y_j^{\#_a}$ in Lemma \ref{bads} and using the fact $\varphi(z^*)=\overline{\varphi(z)}$ for any $z\in\mathfrak{A}$, we obtain the desired result.
\end{proof}

\begin{lemma}\label{moore}
Let $x$ be a regular element of $\mathfrak{A}$ and $\mu\ge 0$.   If $\varphi\in \mathcal{S}(\mathfrak{A})$, then
\begin{eqnarray*}
     \Big |\varphi(x)\Big|^{2}\le \frac{1}{1+\mu} \Big|\varphi(x) \Big|\bigg(\varphi(xx^{\dag})\varphi(x^*x)\bigg)^{\frac{1}{2}}+\frac{\mu}{1+\mu}\varphi(xx^{\dag})\varphi(x^*x).
\end{eqnarray*}
\end{lemma}
\begin{proof} Let  $x$ be a regular element of $\mathfrak{A}$ and $x^{\dag}$ represents the corresponding Moore-penrose inverse. Then  for any $\varphi\in \mathcal{S}(\mathfrak{A})$ we have
    \begin{align*}
       \Big |\varphi(x)\Big|^{2}&=\Big|\varphi(xx^{\dag}x)\Big|^{2}\\
        &= \Big|\varphi((xx^{\dag})^*x)\Big|^{2}\\
        &\le\frac{1}{1+\mu} \Big|\varphi(x) \Big|\bigg(\varphi(xx^{\dag})\varphi(x^*x)\bigg)^{\frac{1}{2}}+\frac{\mu}{1+\mu}\varphi(xx^{\dag})\varphi(x^*x),
    \end{align*} where the above inequality is derived from Lemma \ref{newcauchy}. Hence, the result follows.
\end{proof}

\section{\textbf{Some upper bounds of $a$-numerical radius}}\label{sec3} 
In this section,  several upper bounds for the $a$-numerical
radius involving $a$-Crawford number are derived, which refine some existing ones.
\begin{theorem}\label{bbvvm} Let $x_j\in \mathfrak{A}_a$, $j=1,2,\cdots, k$ and $\mu\ge 0$.   If $\textbf{P}=(p_1,p_2,\cdots,p_k)\in \mathcal{P}_k$ is a probability distribution, then   
    \begin{align*}
         v^{4}_a\Big(\sum_{j=1}^k p_jx_j\Big)\le&
\frac{1+2\mu}{16(1+\mu)}\bigg(\sum_{j=1}^{k}p_j\left\|x^{\#_a}_jx_j+x_jx^{\#_a}_j\right\|_a\bigg)^2-\frac{1}{16}c^2_{a}\bigg(\sum_{j=1}^{k}p_j\left(x^{\#_a}_jx_j-x_jx^{\#_a}_j\right)\bigg)\nonumber\\&+\frac{3+2\mu}{8(1+\mu)}v_a\Big(\sum_{j=1}^{k}p_jx^2_j\Big)\Bigg(\bigg(\sum_{j=1}^{k}p_j\left\|x^{\#_a}_jx_j+x_jx^{\#_a}_j\right\|_a\bigg)^2\\&-c^2_{a}\bigg(\sum_{j=1}^{k}p_j\left(x^{\#_a}_jx_j-x_jx^{\#_a}_j\right)\bigg)\Bigg)^{\frac{1}{2}}\nonumber.
    \end{align*}

\end{theorem}
\begin{proof} 
Let $x_j\in \mathfrak{A}_a$, $j=1,2,\cdots, k$ and $\varphi\in \mathcal{S}_a(\mathfrak{A})$. Then  Corollary \ref{klj1} yields
    \begin{align}
       & \Big|\sum_{j=1}^k p_j\varphi(ax_j)\Big|^{4}\nonumber\\=&\left(\Big|\sum_{j=1}^{k}p_j\varphi(ax_j)\Big|\Big|\sum_{j=1}^{k}p_j\varphi(ax_j)\Big|\right)^{2}\nonumber\\
\le&
\frac{1}{4}\Big(\sum_{j=1}^{k}p_j\varphi(ax^{\#_a}_jx_j)\sum_{j=1}^{k}p_j\varphi(ax_jx^{\#_a}_j)\Big)\nonumber\\&+\frac{3+2\mu}{4(1+\mu)}\Big(\sum_{j=1}^{k}p_j\varphi(ax^{\#_a}_jx_j)\sum_{j=1}^{k}p_j\varphi(ax_jx^{\#_a}_j)\Big)^{\frac{1}{2}}\Big|\sum_{j=1}^{k}p_j\varphi(ax^2_j)\Big|\nonumber\\&+\frac{\mu}{4(1+\mu)}\bigg(\sum_{j=1}^kp_j\frac{\varphi(a(x_jx^{\#_a}_j+x^{\#_a}_jx_j))}{2}\bigg)^{2},\label{p}
    \end{align}
Next, by applying the identity
$4ab=(a+b)^2-(a-b)^2$ for all $a,b\in\mathbb{R}$, we observe that \begin{align}
\nonumber&\sum_{j=1}^{k}p_j\varphi(ax^{\#_a}_jx_j)\sum_{j=1}^{k}p_j\varphi(ax_jx^{\#_a}_j)\\
&=\frac{1}{4}\bigg(\sum_{j=1}^{k}p_j\varphi\left(a(x^{\#_a}_jx_j+x_jx^{\#_a}_j)\right)\bigg)^2-\frac{1}{4}\bigg(\sum_{j=1}^{k}p_j\varphi\left(a(x^{\#_a}_jx_j-x_jx^{\#_a}_j)\right)\bigg)^2\nonumber\\
&\le \frac{1}{4}\bigg(\sum_{j=1}^{k}p_j\left\|x^{\#_a}_jx_j+x_jx^{\#_a}_j\right\|_a\bigg)^2-\frac{1}{4}c^2_{a}\bigg(\sum_{j=1}^{k}p_j\left(x^{\#_a}_jx_j-x_jx^{\#_a}_j\right)\bigg).\label{q}
    \end{align}
Now, from the inequalities (\ref{p}) and (\ref{q}), it follows that
\begin{align*}
    &\Big|\varphi\Big(a\Big(\sum_{j=1}^k p_jx_j\Big)\Big)\Big|^{4}\\=&\Big|\sum_{j=1}^k p_j\varphi(ax_j)\Big|^{4}\\\le&
\frac{1+2\mu}{16(1+\mu)}\bigg(\sum_{j=1}^{k}p_j\left\|x^{\#_a}_jx_j+x_jx^{\#_a}_j\right\|_a\bigg)^2-\frac{1}{16}c^2_{a}\bigg(\sum_{j=1}^{k}p_j\left(x^{\#_a}_jx_j-x_jx^{\#_a}_j\right)\bigg)\nonumber\\&+\frac{3+2\mu}{8(1+\mu)}v_a\Big(\sum_{j=1}^{k}p_jx^2_j\Big)\Bigg(\bigg(\sum_{j=1}^{k}p_j\left\|x^{\#_a}_jx_j+x_jx^{\#_a}_j\right\|_a\bigg)^2\\&-c^2_{a}\bigg(\sum_{j=1}^{k}p_j\left(x^{\#_a}_jx_j-x_jx^{\#_a}_j\right)\bigg)\Bigg)^{\frac{1}{2}}\nonumber.
 \end{align*}
Therefore, by taking the supremum over all $\varphi\in \mathcal{S}_a(\mathfrak{A})$, we arrive at the required result.

\end{proof}
Taking $k=1$ in Theorem \ref{bbvvm}  yields the following corollary.

\begin{cor}\label{cor01}
      Let $\mathfrak{A} $ be a unital ${C}^*$-algebra and $\mu\ge 0$.   Then  
 
    \begin{align*}
         v^{4}_a\Big(x\Big)\le&
\frac{1+2\mu}{16(1+\mu)}\left\|x^{\#_a}x+xx^{\#_a}\right\|_a^2-\frac{1}{16}c^2_{a}\bigg(x^{\#_a}x-xx^{\#_a}\bigg)\nonumber\\&+\frac{3+2\mu}{8(1+\mu)}v_a\Big(x^2\Big)\Bigg(\left\|x^{\#_a}x+xx^{\#_a}\right\|_a^2-c^2_{a}\bigg(x^{\#_a}x-xx^{\#_a}\bigg)\Bigg)^{\frac{1}{2}}\nonumber.
    \end{align*}
   
\end{cor}

\begin{remark}$(i)$  Since $v_a(x^2)\le v_a^2(x)\le \frac{1}{2}\|x^\#_ax+xx^\#_a \|_a$ and $c^2_{a}\big(x^{\#_a}x-xx^{\#_a}\big)\ge 0$, it follows that
 \begin{align*}&\frac{1+2\mu}{16(1+\mu)}\left\|x^{\#_a}x+xx^{\#_a}\right\|_a^2-\frac{1}{16}c^2_{a}\bigg(x^{\#_a}x-xx^{\#_a}\bigg)\nonumber\\&+\frac{3+2\mu}{8(1+\mu)}v_a\Big(x^2\Big)\Bigg(\left\|x^{\#_a}x+xx^{\#_a}\right\|_a^2-c^2_{a}\bigg(x^{\#_a}x-xx^{\#_a}\bigg)\Bigg)^{\frac{1}{2}}\\
        \le&\frac{1+2\mu}{16(1+\mu)}\Big\|x^{\#_a}x+xx^{\#_a}\Big\|^{2}_a+ \frac{3+2\mu}{8(1+\mu)}\Big\|x^{\#_a}x+xx^{\#_a}\Big\|_av_a\big(x^2\big)\\
            \le& \frac{1}{4}\Big\|x^{\#_a}x+xx^{\#_a}\Big\|_a^2.
    \end{align*}

Therefore,  Corollary \ref{cor01} provides a refinement of the inequality $(\ref{in33})$.\\

     $(ii)$
   According to  \cite[Th. 3.1]{hai}, for every $T\in\mathscr{B}_A(\mathcal{H}) $, the following inequality holds:
     \begin{eqnarray}
         w^{4}_A\left(T\right)\le \frac{3+2\mu}{8(1+\mu)}\Big\|T^{\#_A}T+TT^{\#_A}\Big\|w_A\big(T^2\big)+\frac{1+2\mu}{16(1+\mu)}\Big\|T^{\#_A}T+TT^{\#_A}\Big\|^{2}. \label{qqww1}
    \end{eqnarray} Since $c^2_{a}\bigg(x^{\#_a}x-xx^{\#_a}\bigg)\ge 0$, the estimate in Corollary \ref{cor01} involves additional subtractive terms, which further decrease the right-hand side.
Hence, by taking $\mathfrak{A}=\mathscr{B}(\mathcal{H})$, it follows that  Corollary \ref{cor01} gives a sharper estimate than  (\ref{qqww1}).

\end{remark}
\begin{example}
It was shown in \cite{Alizamani} that for any  $x\in \mathfrak{A}$ and $\mu\ge0$,
    \begin{align}
        v^{4}(x)\le \frac{1}{4(1+\mu)}v(x^2)\|x^*x+xx^*\|+\frac{1+2\mu}{4(1+\mu)}\|(x^*x)^2+(xx^*)^2\|\label{Alizamaniexample}
    \end{align} Consider $\mathfrak{A}=\mathscr{M}_2(\mathbb{C)}$,  the algebra of all $2\times 2$ complex matrices, and take $a=\textbf{1}_{\mathfrak{A}}$.  Now, for the matrix $x=\begin{bmatrix}
        0 & 1\\
        2 & 0
    \end{bmatrix}$, the  inequality (\ref{Alizamaniexample}) yields
     $v^4(x)\le \frac{27+34\mu}{4(1+\mu)}$. On the other hand, Corollary (\ref{cor01}) 
 provides the estimate  $v^4(x)\le \frac{64+73\mu}{16(1+\mu)}$.  Since $\frac{64+73\mu}{16(1+\mu)}< \frac{27+34\mu}{4(1+\mu)}$ for all $\mu\ge 0$, it follows that Corollary (\ref{cor01}) gives a sharper  upper bound compared to inequality (\ref{Alizamaniexample}).\end{example}

\begin{theorem}\label{0stproduct}
Let $x_j$, $y_j\in \mathfrak{A}_a$, $j=1,2,\cdots, k$ and $\mu\ge 0$.    If $\textbf{P}=(p_1,p_2,\cdots,p_k)\in \mathcal{P}_k$ is a probability distribution, then
\begin{align*}
   & v^{4}_a\left(\sum_{j=1}^k p_jx^{\#_a}_jy_j\right)\\\le&\frac{1+2\mu}{16(1+\mu)} \bigg(\sum_{j=1}^{k}p_j\Big\|(x^{\#_a}_jx_j)^2+(y^{\#_a}_jy_j)^2\Big\|_a\bigg)^2-\frac{1}{4}c_{a}^2\bigg(\sum_{j=1}^{k}p_j\big((x^{\#_a}_jx_j)^2-(y^{\#_a}_jy_j)^2\big)\bigg)\nonumber\\&+\frac{3+2\mu}{8(1+\mu)} v_a^r\Big(\sum_{j=1}^{k}p_jy^{\#_a}_jy_jx^{\#_a}_jx_j\Big)\bigg(\bigg(\sum_{j=1}^{k}p_j\Big\|(x^{\#_a}_jx_j)^2+(y^{\#_a}_jy_j)^2\Big\|_a\bigg)^2\\
&-c_{a}^2\bigg(\sum_{j=1}^{k}p_j\big((x^{\#_a}_jx_j)^2-(y^{\#_a}_jy_j)^2\big)\bigg)\bigg)^{\frac{1}{2}} .\end{align*}

\end{theorem}
\begin{proof}
Let $x_j, y_j\in \mathfrak{A}_a$, $j=1,2,\cdots, k$ and $\varphi\in \mathcal{S}_a(\mathfrak{A})$. By  Lemma \ref{vv2}, we obtain
     \begin{align*}
        \Big|\sum_{j=1}^kp_j\varphi(x^*_jay_j)\Big|^{4}
         &\le\Big(\sum_{j=1}^kp_j\big|\varphi(x^*_jay_j)\big|\Big)^{4}\nonumber\\
&\le\Big(\sum_{j=1}^kp_j\sqrt{\varphi(x^*_jax_j)}\sqrt{\varphi(y^*_jay_j)}\Big)^{4}\nonumber\\
        &\le \Big(\sum_{j=1}^kp_j\varphi(x^*_jax_j)\sum_{j=1}^kp_j\varphi(y^*_jay_j)\Big)^{2}\nonumber\\
&=\bigg(\Big|\sum_{j=1}^kp_j\varphi(ax^{\#_a}_jx_j)\Big|\Big|\sum_{j=1}^kp_j\varphi(ay^{\#_a}_jy_j)\Big|\bigg)^{2}.\end{align*}
Since, 
$x^{\#_a}_jx_j$ and $y^{\#_a}_jy_j$ belong to $\mathfrak{A}_a$  for each $j=1,2,\cdots,k$, so applying the Corollary \ref{klj1} we have
\begin{align}&  \Big|\sum_{j=1}^kp_j\varphi(x^*_jay_j)\Big|^{4}\nonumber\\
\le&\bigg(\Big|\sum_{j=1}^kp_j\varphi(ax^{\#_a}_jx_j)\Big|\Big|\sum_{j=1}^kp_j\varphi(ay^{\#_a}_jy_j)\Big|\bigg)^{2}\nonumber\\\le&\frac{1}{4} \bigg(\sum_{j=1}^{k}p_j\varphi\big(a(x^{\#_a}_jx_j)^2\big)\sum_{j=1}^{k}p_j\varphi\big(a(y^{\#_a}_jy_j)^2\big)\bigg)\nonumber\\&\nonumber+\frac{3+2\mu}{4(1+\mu)} \bigg(\sum_{j=1}^{k}p_j\varphi\big(a(x^{\#_a}_jx_j)^2\big)\sum_{j=1}^{k}p_j\varphi\big(a(y^{\#_a}_jy_j)^2\big)\bigg)^{\frac{1}{2}}\Big|\sum_{j=1}^{k}p_j\varphi(ay^{\#_a}_jy_jx^{\#_a}_jx_j)\Big|\\&+\frac{\mu}{4(1+\mu)}\bigg(\sum_{j=1}^kp_j\frac{\varphi\big(a\big((x^{\#_a}_jx_j)^2+(y^{\#_a}_jy_j)^2\big)\big)}{2}\bigg)^{2}.\label{A}
\end{align}
Next, it follows from the identity
$4ab=(a+b)^2-(a-b)^2$ for all $a,b\in\mathbb{R}$, that
\begin{align}
&\sum_{j=1}^{k}p_j\varphi\big(a(x^{\#_a}_jx_j)^2\big)\sum_{j=1}^{k}p_j\varphi\big(a(y^{\#_a}_jy_j)^2\big)\nonumber\\& =\frac{1}{4}\bigg(\sum_{j=1}^{k}p_j\varphi\Big(a\big((x^{\#_a}_jx_j)^2+(y^{\#_a}_jy_j)^2\big)\Big)\bigg)^2- \frac{1}{4}\bigg(\sum_{j=1}^{k}p_j\varphi\Big(a\big((x^{\#_a}_jx_j)^2-(y^{\#_a}_jy_j)^2\big)\Big)\bigg)^2\nonumber\\
& \le \frac{1}{4}\bigg(\sum_{j=1}^{k}p_j\Big\|(x^{\#_a}_jx_j)^2+(y^{\#_a}_jy_j)^2\Big\|_a\bigg)^2-\frac{1}{4}c_{a}^2\bigg(\sum_{j=1}^{k}p_j\big((x^{\#_a}_jx_j)^2-(y^{\#_a}_jy_j)^2\big)\bigg)
\label{AB}.  \end{align}
Now, from the inequalities (\ref{A}) and (\ref{AB}), it implies that
\begin{align*}&\Big|\varphi\Big(a\Big(\sum_{j=1}^kp_jx^{\#_a}_jy_j\Big)\Big)\Big|^{4}\\
=&\Big|\sum_{j=1}^kp_j\varphi(x^*_jay_j)\Big|^{4}\\ \le&\frac{1+2\mu}{16(1+\mu)} \bigg(\sum_{j=1}^{k}p_j\Big\|(x^{\#_a}_jx_j)^2+(y^{\#_a}_jy_j)^2\Big\|_a\bigg)^2-\frac{1}{4}c_{a}^2\bigg(\sum_{j=1}^{k}p_j\big((x^{\#_a}_jx_j)^2-(y^{\#_a}_jy_j)^2\big)\bigg)\nonumber\\&+\frac{3+2\mu}{8(1+\mu)} v_a^r\Big(\sum_{j=1}^{k}p_jy^{\#_a}_jy_jx^{\#_a}_jx_j\Big)\bigg(\bigg(\sum_{j=1}^{k}p_j\Big\|(x^{\#_a}_jx_j)^2+(y^{\#_a}_jy_j)^2\Big\|_a\bigg)^2\\
&-c_{a}^2\bigg(\sum_{j=1}^{k}p_j\big((x^{\#_a}_jx_j)^2-(y^{\#_a}_jy_j)^2\big)\bigg)\bigg)^{\frac{1}{2}} .
\end{align*} Finally, 
by taking the supremum over all $\varphi\in \mathcal{S}_a(\mathfrak{A})$, we arrive at the required result.

\end{proof}
Taking $k = 1$  in Theorem \ref{0stproduct}, yields the following corollary.
\begin{cor}\label{cor2} 
Let $\mathfrak{A} $ be a unital ${C}^*$-algebra and $\mu\ge 0$.  If $x,y\in \mathfrak{A}_a$, then 
\begin{align*}
v^{4}_a\left(x^{\#_a}y\right)&\le\frac{1+2\mu}{16(1+\mu)}\|(x^{\#_a}x)^2+(y^{\#_a}y)^2\|_a^2-\frac{1}{4}c_{a}^2\big((x^{\#_a}x)^2-(y^{\#_a}y)^2\big)\\&+\frac{3+2\mu}{8(1+\mu)}v_a\big( y^{\#_a}yx^{\#_a}x\big)\Bigg(\|(x^{\#_a}x)^2+(y^{\#_a}y)^2\|_a^2-c_{a}^2\big((x^{\#_a}x)^2-(y^{\#_a}y)^2\big)\Bigg)^\frac{1}{2}.
\end{align*}

\end{cor}
\begin{remark}
  Since $v_a\Big(y^{\#_a}yx^{\#_a}x\Big)=v_a\Big(x^{\#_a}xy^{\#_a}y\Big)$ and $ c_{a}^2\big((x^{\#_a}x)^2-(y^{\#_a}y)^2\big)\ge 0$, we observe that, in the special case where
 $\mathfrak{A}=\mathscr{B}(\mathcal{H})$, Corollary \ref{cor2} improves  the upper bound proved in \cite[Theorem 3.3]{hai} and it provides an improvement of \cite[Th. 3.8]{hq02}.\\


\end{remark}
\begin{example}
   In \cite{pintu}, it was shown that for elements $x,y,z\in \mathfrak{A}$ and $r\in \mathbb{N}$, 
    \begin{align*}
        v^{r}\Big(\sum_{j=1}^k x_j^*z_jy_j\Big)\le \frac{n^{r-1}}{\sqrt{2}}v\Big(\sum_{j=1}^k\big((y_j^*|z_j|y_j)^r+i(x^*_j|z_j^*|x_j)^r\big)\Big).
    \end{align*}  Taking $k=1$,  $r=4$ and $z=\mathbf{1}_{\mathfrak{A}}$, this reduces to
     \begin{align}
        v^{4}\big(x^*y\big)\le \frac{1}{\sqrt{2}}v\Big((y^*y)^4+i(x^*x)^4\Big).\label{pintuexample}
    \end{align}  
      Now, let $\mathfrak{A}=\mathscr{M}_2(\mathbb{C)}$ and choose $a=\textbf{1}_{\mathfrak{A}}$. Consider the matrices $x=\begin{pmatrix}
           2 & 0\\ 0 & 1
       \end{pmatrix}$ and  $y=\begin{pmatrix}
           1 & 0\\ 0 & 0
       \end{pmatrix}.$  For these choices, inequality \eqref{pintuexample} yields $v^{4}(x^*y)\le \sqrt{32768}$. On the other hand, Corollary \ref{cor2} provides the estimate \[v^{4}(x^*y)\le \frac{289(1+2\mu)}{16(1+\mu)}-\frac{1}{4}+\frac{3\sqrt{2}(3+2\mu)}{(1+\mu)},\] which depends on a positive parameter $\mu$.  For any choice of $\mu$, the value of this above expression does not exceed $44.5$. Therefore, Corollary \ref{cor2} yields a significantly sharper upper bound than inequality \eqref{pintuexample}.\end{example}

\begin{theorem}\label{lamma}  Let $x_j\in \mathfrak{A}_a$, $j=1,2,\cdots, k$ and $\mu\ge 0$. If $\textbf{P}=(p_1,p_2,\cdots,p_k)\in \mathcal{P}_k$ is a probability distribution, then 
 \begin{align*}
      v^{4}_a\bigg(\sum_{j=1}^k p_jx_j\bigg) \le &\frac{1}{2(1+\mu)} v_a^2\Big(\sum_{j=1}^{k}p_jx_j^2 \Big)+\frac{1+2\mu}{8(1+\mu)}\bigg(\sum_{j=1}^{k}p_j\|x^{\#_a}_jx_j+x_jx^{\#_a}_j\|_a\bigg)^2\\&-\frac{1}{8}c_a^2\bigg(\sum_{j=1}^{k}p_j(x^{\#_a}_jx_j-x_jx^{\#_a}_j)\bigg).
\end{align*} 
\end{theorem}
\begin{proof} 
 Let $x_j\in \mathfrak{A}_a$, $j=1,2,\cdots, k$ and $\varphi\in \mathcal{S}_a(\mathfrak{A})$.
Then  from Corollary \ref{2ndcor} we have,
\begin{align}\label{Bl}
  & \nonumber \Big|\sum_{j=1}^k p_j\varphi(ax_j)\Big|^{4}
\\&=\bigg(\Big|\sum_{j=1}^{k}p_j\varphi(ax_j)\Big|\Big|\sum_{j=1}^{k}p_j\varphi(ax_j)\Big|\bigg)^{2}\nonumber\\\nonumber
        &\le \frac{1}{2(1+\mu)} \Big|\sum_{j=1}^{k}p_j\varphi(ax_j^2) \Big|^{2}+ \frac{\mu}{8(1+\mu)}\bigg(\sum_{j=1}^{k}p_j\varphi\big(a(x^{\#_a}_jx_j+x_jx^{\#_a}_j)\big)\bigg)^{2}\\&+  \frac{1}{2} \Big(\sum_{j=1}^kp_j\varphi(ax^{\#_a}_jx_j)\sum_{j=1}^kp_j  \varphi(ax_jx^{\#_a}_j)\Big).
 \end{align}
Now, using the identity
$4ab=(a+b)^2-(a-b)^2$ for all $a,b\in\mathbb{R}$, it follows that
\begin{align}
&\sum_{j=1}^kp_j\varphi(ax^{\#_a}_jx_j)\sum_{j=1}^kp_j  \varphi(ax_jx^{\#_a}_j)\nonumber\\=&\frac{1}{4}\bigg(\sum_{j=1}^{k}p_j\varphi\big(a(x^{\#_a}_jx_j+x_jx^{\#_a}_j)\big)\bigg)^2-\frac{1}{4}\bigg(\sum_{j=1}^{k}p_j\varphi\big(a(x^{\#_a}_jx_j-x_jx^{\#_a}_j)\big)\bigg)^2\nonumber\\
\le &\frac{1}{4}\bigg(\sum_{j=1}^{k}p_j\|x^{\#_a}_jx_j+x_jx^{\#_a}_j\|_a\bigg)^2-\frac{1}{4}c_a^2\bigg(\sum_{j=1}^{k}p_j(x^{\#_a}_jx_j-x_jx^{\#_a}_j)\bigg).\label{Bl2}
\end{align}
Now, from the inequalities  \eqref{Bl} and \eqref{Bl2}, it follows that
 \begin{align*}&\Big|\varphi\Big(a\Big(\sum_{j=1}^k p_jx_j\Big)\Big)\Big|^{4}\\
=&\Big|\sum_{j=1}^k p_j\varphi(ax_j)\Big|^{4}\\\le& \frac{1}{2(1+\mu)} v_a^2\Big(\sum_{j=1}^{k}p_jx_j^2 \Big)+\frac{1+2\mu}{8(1+\mu)}\bigg(\sum_{j=1}^{k}p_j\|x^{\#_a}_jx_j+x_jx^{\#_a}_j\|_a\bigg)^2\\&-\frac{1}{8}c_a^2\bigg(\sum_{j=1}^{k}p_j(x^{\#_a}_jx_j-x_jx^{\#_a}_j)\bigg).
    \end{align*}
By taking the supremum over all $\varphi\in \mathcal{S}_a(\mathfrak{A})$, we arrive at the desired result.

\end{proof} The following result follows by substituting $k = 1$ in Theorem \ref{lamma}.
\begin{cor} \label{cor3} Let $x\in \mathfrak{A}_a$ and $\mu\ge 0$. Then
\begin{align*}
    v_a^{4}\left(x\right) \le \frac{1}{2(1+\mu)}  v_a^2\big(x^2\big)+ \frac{1+2\mu}{8(1+\mu)} \big\|x^{\#_a}x+xx^{\#_a}\big\|_a^2-\frac{1}{8}c_a^2\big(x^{\#_a}x-xx^{\#_a}\big).
\end{align*}

\end{cor}
\begin{remark} $(i)$
 Since $v_a(x^2)\le v_a^2(x)$ and $c_a^2\big(x^{\#_a}x-xx^{\#_a}\big)\ge 0$, it follows that
 \begin{align*}
     &\frac{1}{2(1+\mu)}v_a^{2}\big(x^2\big)
  +\frac{1+2\mu}{8(1+\mu)}\Big\|x^{\#_a}x+xx^{\#_a}\Big\|_a^2-\frac{1}{8} c_a^2\big(x^{\#_a}x-xx^{\#_a}\big)\\
  \le& \frac{1}{4}\Big\|x^{\#_a}x+xx^{\#_a}\Big\|_a^2.
 \end{align*}
    Therefore, Corollary \ref{cor3} provides a refinement of the inequality $(\ref{in33})$.\\ 
  
  $(ii)$
    In particular, if we take $\mathfrak{A}=\mathscr{B}(\mathcal{H})$ then the Corollary \ref{cor3} refines the upper bound bouned
   obtained in \cite[Theorem 3.2]{hai}, due to the fact that $c_a^2\big(x^{\#_a}x-xx^{\#_a}\big)\ge 0$ and further it  refines the bound in \cite[Corollary 3.6]{hq02}.
   
\end{remark}
\begin{example}
In \cite{Alizamani}, it was shown that if $x\in \mathfrak{A}$ and $\mu\in[0,1]$, then
  \begin{align}
    v^4(x)\le \frac{1+\mu}{4}  \|(xx^*)^2+(x^*x)^2\|+\frac{1-\mu}{2}v^2(x^2).\label{Alizamaniexample2}
  \end{align}
   Let $\mathfrak{A}=\mathscr{M}_2(\mathbb{C)}$ and take $a=\textbf{1}_{\mathfrak{A}}$.  Consider  the matrix $x=\begin{bmatrix}
        0 & 1\\
        z & 0
    \end{bmatrix}$ where $z\in \mathbb{C}$. For this choice of $x$, the right-hand side of inequality \eqref{Alizamaniexample2} becomes to \[\frac{1+\mu}{4}(1+|z|^4)  +\frac{1-\mu}{2}|z|^2\,\,\,(=Q_1). \] On the  other hand,  applying Corollary \ref{cor3} yields the estimate \[\frac{|z|^2}{2(1+\mu)}  + \frac{1+2\mu}{8(1+\mu)} (1+|z|^2)^2-\frac{(|z|^2-1)^2}{8}\,\,\,(=Q_2).\]  A direct comparison shows that $Q_2<Q_1$ for all $\mu\in[0,1]$. Therefore, Corollary \ref{cor3} provides a sharper upper bound than the estimate given in \eqref{Alizamaniexample2}.
\end{example}

\begin{theorem}\label{1stproduct}
Let $x_j$, $y_j\in \mathfrak{A}_a$, $j=1,2,\cdots, k$ and $\mu\ge 0$.  If $\textbf{P}=(p_1,p_2,\cdots,p_k)\in \mathcal{P}_k$ is a probability distribution, then 
\begin{align*}
   &  v^{4}_a\bigg(\sum_{j=1}^k p_jx^{\#_a}_jy_j\bigg) \\\le& \frac{1}{2(1+\mu)} v_a^{2}\big(\sum_{j=1}^{k}p_jy^{\#_a}_jy_jx^{\#_a}_jx_j\big)+\frac{1+2\mu}{8(1+\mu)}\Big(\sum_{j=1}^{k}p_j\Big\|(x^{\#_a}_jx_j)^2+(y^{\#_a}_jy_j)^2\Big\|_a\Big)^{2}\\
&-\frac{1}{8}c_{a}^2\bigg(\sum_{j=1}^{k}p_j\big((x^{\#_a}_jx_j)^2-(y^{\#_a}_jy_j)^2\big)\bigg).
\end{align*}
\end{theorem}

\begin{proof} 
 Let $x_j\in \mathfrak{A}_a$, $j=1,2,\cdots, k$ and $\varphi\in \mathcal{S}_a(\mathfrak{A})$.  By  Lemma \ref{vv2}, we obtain
    \begin{align*}
  \Big|\sum_{j=1}^kp_j\varphi(x^*_jay_j)\Big|^{4}\nonumber
&\le\Big(\sum_{j=1}^kp_j|\varphi(x^*_jay_j)|\Big)^{4}\nonumber\\\nonumber
&\le\Big(\sum_{j=1}^kp_j\sqrt{\varphi(x^*_jax_j)}\sqrt{\varphi(y^*_jay_j)}\Big)^{4}\\\nonumber
        &\le \Big(\sum_{j=1}^kp_j\varphi(x^*_jax_j)\sum_{j=1}^kp_j\varphi(y^*_jay_j)\Big)^{2}\\      \nonumber   &=\bigg(\Big|\sum_{j=1}^kp_j\varphi(ax^{\#_a}_jx_j)\Big|\Big|\sum_{j=1}^kp_j\varphi(ay^{\#_a}_jy_j)\Big|\bigg)^{2}.\end{align*}
Since, 
$x^{\#_a}_jx_j$ and $y^{\#_a}_jy_j$ belong to $\mathfrak{A}_a$  for each $j=1,2,\cdots,k$, so applying the Corollary \ref{2ndcor} we have
        \begin{align}
      & \Big|\sum_{j=1}^kp_j\varphi(x^*_jay_j)\Big|^{4} \nonumber\\ &\le\bigg(\Big|\sum_{j=1}^kp_j\varphi(ax^{\#_a}_jx_j)\Big|\Big|\sum_{j=1}^kp_j\varphi(ay^{\#_a}_jy_j)\Big|\bigg)^{2}\nonumber\\ \le &\frac{1}{2(1+\mu)}\Big|\sum_{j=1}^{k}p_j\varphi(ay^{\#_a}_jy_jx^{\#_a}_jx_j)\Big|^{2}
  +\frac{\mu}{8(1+\mu)} \bigg(\sum_{j=1}^{k}p_j\varphi\big(a\big((x^{\#_a}_jx_j)^2+(y^{\#_a}_jy_j)^2\big)\big)\bigg)^{2}\nonumber\\&+\frac{1}{2}\bigg(\sum_{j=1}^kp_j\varphi\big(a(x^{\#_a}_jx_j)^2\big)\sum_{j=1}^kp_j  \varphi\big(a(y^{\#_a}_jy_j)^2\big)\bigg). \label{C}
    \end{align}
Again, the identity
$4ab=(a+b)^2-(a-b)^2$ for all $a,b\in\mathbb{R}$ yields
\begin{align}
&\sum_{j=1}^kp_j\varphi\big(a(x^{\#_a}_jx_j)^2\big)\sum_{j=1}^kp_j  \varphi\big(a(y^{\#_a}_jy_j)^2\big)\nonumber \\&= \frac{1}{4} \bigg(\sum_{j=1}^{k}p_j\varphi\Big(a\big((x^{\#_a}_jx_j)^2+(y^{\#_a}_jy_j)^2\big)\Big)\bigg)^2-\frac{1}{4} \bigg(\sum_{j=1}^{k}p_j\varphi\Big(a\big((x^{\#_a}_jx_j)^2-(y^{\#_a}_jy_j)^2\big)\Big)\bigg)^2 \nonumber\\&\le \frac{1}{4}\bigg(\sum_{j=1}^{k}p_j\Big\|(x^{\#_a}_jx_j)^2+(y^{\#_a}_jy_j)^2\Big\|_a\bigg)^2-\frac{1}{4} c_{a}^2\bigg(\sum_{j=1}^{k}p_j\big((x^{\#_a}_jx_j)^2-(y^{\#_a}_jy_j)^2\big)\bigg).\label{ca}
\end{align}Now, from the inequalities (\ref{C}) and (\ref{ca}), it follows that
\begin{align*}&\Big|\varphi\Big(a\Big(\sum_{j=1}^kp_jx^{\#_a}_jy_j\Big)\Big)\Big|^{4}\\
&=\Big|\sum_{j=1}^kp_j\varphi(x^*_jay_j)\Big|^{4}\\  \le& \frac{1}{2(1+\mu)} v_a^{2}\big(\sum_{j=1}^{k}p_jy^{\#_a}_jy_jx^{\#_a}_jx_j\big)+\frac{1+2\mu}{8(1+\mu)}\Big(\sum_{j=1}^{k}p_j\Big\|(x^{\#_a}_jx_j)^2+(y^{\#_a}_jy_j)^2\Big\|_a\Big)^{2}\\
&-\frac{1}{8}c_{a}^2\bigg(\sum_{j=1}^{k}p_j\big((x^{\#_a}_jx_j)^2-(y^{\#_a}_jy_j)^2\big)\bigg).
\end{align*}
By taking the supremum over all $\varphi\in \mathcal{S}_a(\mathfrak{A})$, we get the required result.
\end{proof} Choosing  $k = 1$ in Theorem \ref{1stproduct} yields the following corollary.
\begin{cor}\label{cor4}   Let $x,y\in \mathfrak{A}_a$ and $\mu\ge 0$. Then 
   \begin{align*}  v_a^{4}\left(x^{\#_a}y\right)&\le
\frac{1}{2(1+\mu)} v^{2}_a\big(y^{\#_a}yx^{\#_a}x\big)
  +\frac{1+2\mu}{8(1+\mu)}\big\|(x^{\#_a}x)^2+(y^{\#_a}y)^2\big\|_a^{2}\\&-\frac{1}{8}c_{a}^2\big((x^{\#_a}x)^2-(y^{\#_a}y)^2\big).
    \end{align*}
\end{cor}
\begin{remark}
$(i)$ by choosing the function
  $\phi(t)=t^2$ for $t\ge 0$ in \cite[Th. 2.30]{smpatra},   we deduce that for any   $x,y,z\in\mathfrak{A}$ with $\|z\|\le 1$,
\begin{eqnarray}
    v^2(x^*zy)\le \frac{\|z\|}{2}\|(x^*x)^2+(y^*y)^2\|.\label{ll02}
\end{eqnarray}

Again, for any $x,y\in \mathfrak{A}$, we have
\begin{eqnarray}
   v(y^*yx^*x)&=&\sup_{\varphi\in\mathcal{S}\mathfrak{(A)}} \big|\varphi(y^*yx^*x)|\nonumber\\
   &\le&  \sup_{f\in\mathcal{S}\mathfrak{(A)}}\sqrt{\varphi\big((y^*y)^2\big)}\sqrt{\varphi\big((x^*x)^2\big)}\nonumber\\
    &\le & \frac{1}{2}\sup_{\varphi\in\mathcal{S}\mathfrak{(A)}}\varphi\big((x^*x)^2+(y^*y)^2\big)\nonumber\\
    &=&\frac{1}{2}\big\|(x^*x)^2+(y^*y)^2\big\|\label{ll}.
\end{eqnarray} 
Therefore,   due to the inclusion of Crawford number in Corollary \ref{cor4} together with the inequality 
 (\ref{ll}), it follows that  for $a=\textbf{1}_{\mathfrak{A}}$ Corollary \ref{cor4} provides a sharper upper bound than (\ref{ll02})  when $z=\textbf{1}_{\mathfrak{A}}$.\\
$(ii)$ Since $v_a\Big(y^{\#_a}yx^{\#_a}x\Big)=v_a\Big(x^{\#_a}xy^{\#_a}y\Big)$ and $ c_{a}^2\big((x^{\#_a}x)^2-(y^{\#_a}y)^2\big)\ge 0$, it follows that, in particular, when we take $\mathfrak{A}=\mathscr{B}(\mathcal{H})$, Corollary \ref{cor4} gives a sharper estimate than the bound obtained in \cite[Th. 3.4]{hai} and it provides an improvement of \cite[Th. 3.8]{hq02}.\\
  
 \end{remark}

\section{\textbf{Some upper bounds of numerical radius}}\label{sec4}
In this section, we derive some numerical radius inequalities of the elements in  $\mathfrak{A}$ using Moore-Penrose inverse.
\begin{theorem}\label{jjhhgg}
   Let $x$ be a regular element of $\mathfrak{A}$ and $\mu\ge 0$. Then
    \begin{eqnarray*}
        v^2(x)&\le&\frac{1}{2(1+\mu)} v\big(x\big) \big\|xx^{\dag}+x^*x\big\|+\frac{\mu}{4(1+\mu)}\big\|xx^{\dag}+x^*x\big\|^2\\
  &\le& \frac{1}{4}\big\|x^*x+xx^{\dag}\big\|^2.
    \end{eqnarray*}

\end{theorem}
\begin{proof} Let $x$ be a regular element of $\mathfrak{A}$ and $\varphi\in \mathcal{S}(\mathfrak{A})$. Then from Lemma \ref{moore} we have,
    \begin{eqnarray*}
           \Big |\varphi(x)\Big|^{2}
           &\le&\frac{1}{1+\mu} \Big|\varphi(x) \Big|\bigg(\varphi(xx^{\dag})\varphi(x^*x)\bigg)^{\frac{1}{2}}+\frac{\mu}{1+\mu}\bigg((\varphi(xx^{\dag})\varphi(x^*x))^{\frac{1}{2}}\bigg)^2\\
   &\le&\frac{1}{2(1+\mu)} \Big|\varphi(x) \Big|\bigg(\varphi\big(xx^{\dag}+x^*x\big)\bigg)+\frac{\mu}{4(1+\mu)}\Bigg(\varphi\Big(xx^{\dag}+x^*x\Big)\Bigg)^2\\
  &\le&\frac{1}{2(1+\mu)} v\big(x\big) \big\|xx^{\dag}+x^*x\big\|+\frac{\mu}{4(1+\mu)}\big\|xx^{\dag}+x^*x\big\|^2.\\
    \end{eqnarray*} By taking the supremum over $\varphi \in \mathcal{S}(\mathfrak{A})$, we arrive at the first desired inequality.

       For the second inequality we have,
       \begin{eqnarray*}
           v(x)&=& \sup_{\varphi\in\mathcal{S}(\mathfrak{A})}\big|\varphi(x)\big|\\
           &=& \sup_{\varphi\in\mathcal{S}(\mathfrak{A})}\big|\varphi(xx^{\dag}x)\big|\\
           &\le & \sup_{\varphi\in\mathcal{S}(\mathfrak{A})}\Big(\varphi(xx^{\dag})\varphi(x^*x)\Big)^{\frac{1}{2}}\\
               &\le & \frac{1}{2}\sup_{\varphi\in\mathcal{S}(\mathfrak{A})}\varphi\Big(xx^{\dag}+x^*x\Big)\\
               &=&  \frac{1}{2}v\big(xx^{\dag}+x^*x\big)= \frac{1}{2}\big\|xx^{\dag}+x^*x\big\|.
       \end{eqnarray*} This completes the proof.
\end{proof}
\begin{remark} Let $\mathcal{CR(H)}$ be the collection of all bounded linear operator on the Hilbert space $\mathcal{H}$ having closed ranges. In  \cite[Th. 2.2]{com} the following bound is obtained, i.e.,
\begin{eqnarray}
    w(T)\le \frac{1}{2}\big\| T^*T+TT^{\dag}\big\|\label{eerr},
\end{eqnarray} where $T\in \mathcal{CR(H)}$. Now, if we take $\mathfrak{A}=\mathscr{B}(\mathcal{H})$ then for any $T\in \mathcal{CR(H)} $ Theorem \ref{jjhhgg} gives better bound than inequality (\ref{eerr}).
\end{remark}

\begin{theorem}\label{pbh}
    Let $x_1$ and $x_2$ be two regular elements of $\mathfrak{A}$ and $\mu\ge 0$. Then 
    \begin{eqnarray*}
        v^2(x_1+x_2) &\le&\frac{1}{2(1+\mu)}\Bigg(v(x_1) \Big\|x_1x^{\dag}_1+x^*_1x_1\Big\|+v(x_2) \Big\|x_2x^{\dag}_2+x^*_2x_2\Big\|\Bigg)\\
        &&+\frac{\mu}{2(1+\mu)}\Big\|x_1x^{\dag}_1+x_2x^{\dag}_2+(x^*_1x_1)^2+(x^*_2x_2)^2\Big\|+2v(x_1)v(x_2).
    \end{eqnarray*}
\end{theorem}
\begin{proof} Let $x$ be a regular element of $\mathfrak{A}$ and $\varphi\in \mathcal{S}(\mathfrak{A})$. Then  from Lemma \ref{moore} we have
    \begin{eqnarray*}
      &&\big|\varphi(x_1+x_2)\big|^2\\
      &\le& |\varphi(x_1)|^2+|\varphi(x_2)|^2+2|\varphi(x_1)||\varphi(x_2)|\\
      &\le&\frac{1}{1+\mu} \Bigg(\Big|\varphi(x_1) \Big|\Big(\varphi(x_1x^{\dag}_1)\varphi(x^*_1x_1)\Big)^{\frac{1}{2}}+ \Big|\varphi(x_2) \Big|\Big(\varphi(x_2x^{\dag}_2)\varphi(x^*_2x_2)\Big)^{\frac{1}{2}}\Bigg)\\
        &&+\frac{\mu}{1+\mu}\Bigg(\varphi(x_1x^{\dag}_1)\varphi(x^*_1x_1)+\varphi(x_2x^{\dag}_2)\varphi(x^*_2x_2)\Bigg)+2|\varphi(x_1)||\varphi(x_2)|\\
         &\le&\frac{1}{2(1+\mu)}\Bigg( \Big|\varphi(x_1) \Big|\varphi\big(x_1x^{\dag}_1+x^*_1x_1\big)+\Big|\varphi(x_2) \Big|\varphi\big(x_2x^{\dag}_2+x^*_2x_2\big)\Bigg)\\
        &&+\frac{\mu}{2(1+\mu)}\bigg(\big( \varphi(x_1x^{\dag}_1)\big)^2+\big( \varphi(x_2x^{\dag}_2)\big)^2+\big( \varphi(x^*_1x_1)\big)^2+\big( \varphi(x^*_2x_2)\big)^2\bigg)+2|\varphi(x_1)||\varphi(x_2)|\\
         &\le&\frac{1}{2(1+\mu)}\Bigg( \Big|\varphi(x_1) \Big|\varphi\big(x_1x^{\dag}_1+x^*_1x_1\big)+\Big|\varphi(x_2) \Big|\varphi\big(x_2x^{\dag}_2+x^*_2x_2\big)\Bigg)\\
        &&+\frac{\mu}{2(1+\mu)}\varphi\Big( x_1x^{\dag}_1+x_2x^{\dag}_2+(x^*_1x_1)^2+(x^*_2x_2)^2\Big)+2|\varphi(x_1)||\varphi(x_2)|\\
         &\le&\frac{1}{2(1+\mu)}\Bigg(v(x_1) \Big\|x_1x^{\dag}_1+x^*_1x_1\Big\|+v(x_2) \Big\|x_2x^{\dag}_2+x^*_2x_2\Big\|\Bigg)\\
        &&+\frac{\mu}{2(1+\mu)}\Big\|x_1x^{\dag}_1+x_2x^{\dag}_2+(x^*_1x_1)^2+(x^*_2x_2)^2\Big\|+2v(x_1)v(x_2).
    \end{eqnarray*} 
    For any positive element $z \in \mathfrak{A}$ and any $\varphi \in \mathcal{S}(\mathfrak{A})$, we have $\big(\varphi(z)\big)^2 \leq \varphi\big(z^2\big)$ (see \cite{conv}). This property is used to derive the fourth inequality.
Taking the supremum over $\varphi \in \mathcal{S}(\mathfrak{A})$, yields the desired result.
\end{proof}
\begin{remark}
    Let $\mathcal{CR(H)}$ be the collection of all bounded linear operator on the Hilbert space $\mathcal{H}$ having closed ranges. In particular, if we take $\mathfrak{A}=\mathscr{B}(\mathcal{H})$,  $x_2=\textbf{0}$, $x_1=T\in \mathcal{CR(H)} $ and $\mu=\frac{1}{2}$ in Theorem \ref{pbh} then we get
    \begin{eqnarray}
    w^2(T)\le \frac{1}{3}w(T)\Big\||T|^2+TT^{\dag}\Big\|+\frac{1}{6}\Big\||T|^4+TT^{\dag}\Big\|,   \label{exp}
    \end{eqnarray} 
    which is derived  in \cite[Th. 2.9]{pintb}.
\end{remark}
To show that Theorem \ref{pbh} provides an improvement of the inequality (\ref{exp}), we  give the following example.
 \begin{example}
    Let $T=\begin{bmatrix}
        0 & 1\\
        0 & 0
    \end{bmatrix}$. Applying inequality (\ref{exp}) to this operator gives the bound $w^2(T)\le \frac{1}{3}$. In contrast, for $x_2=0$, $x_1=T$ and $\mathfrak{A}=\mathscr{B}(\mathcal{H})$ in Theorem \ref{pbh} we obtain 
    \begin{eqnarray*}
       w^2(T)&\le& \frac{1+2\mu}{4(1+\mu)}.
    \end{eqnarray*}This expression depends on a positive parameter $\mu$  and for any choice of $\mu<\frac{1}{2}$, we obtain $ \frac{2\mu+1}{4(1+\mu)}< \frac{1}{3}$.  This means that Theorem \ref{pbh} provides a significant refinement over the estimate obtained from inequality (\ref{exp}).
\end{example}
\section{Conclusion}
\textit{ 
We establish several significant upper bounds for the numerical radius and the a-numerical radius of elements in a unital ${C}^*$-algebra by deriving inequalities based on positive linear functionals. Furthermore, by incorporating additional structural information through the Moore–Penrose inverse, we obtain sharper estimates of the numerical radius compared to standard existing bounds.
These results also suggest several promising directions for future research.}

\section*{Declarations}	
\textit{Acknowledgements.}
 Mr. Saikat Mahapatra 
would like to thank the UGC, Govt. of India for the financial support (NTA Ref. No. 211610170555) in the form of fellowship. Mr. Arobinda Ghosh would like to thank the UGC, Govt. of India for the financial support (NTA Ref. No. 211610122278) in the form of fellowship.\\
\textit{Author Contributions:} All the authors contributed equally to this manuscript and reviewed it. \\
 \textit{Data Availability :} No datasets were generated or analysed during the current study. \\
\textit{Conflict of interest:} There is no competing interest.\\


\end{document}